\documentclass[11pt]{amsart}
\usepackage{latexsym,amssymb,amsmath,youngtab}
\usepackage{youngtab}
\usepackage{tikz}
\usetikzlibrary{matrix}
\usepackage{amsmath,amsthm,amssymb,amscd,MnSymbol}
\usepackage[mathscr]{eucal}
\usepackage{verbatim}
\usepackage{color}
\newtheoremstyle{custom}
  {3pt}
  {3pt}
  {\slshape}
  {}
  {\bfseries}
  {.}
  { }
   {}
\theoremstyle{custom}
\newtheorem{theorem}{Theorem}[section]
\newtheorem{proposition}[theorem]{Proposition}

\newtheorem{proposition/definition}[theorem]{Proposition/Definition}
\newtheorem{lemma}[theorem]{Lemma}
\newtheorem{corollary}[theorem]{Corollary}
\newtheorem{conjecture}[theorem]{Conjecture}

\theoremstyle{definition}
\newtheorem{definition}[theorem]{Definition}

\newtheorem{example}[theorem]{Example}

\newtheorem{question}[theorem]{Question}

\theoremstyle{remark}
\newtheorem{remark}[theorem]{Remark}

\newcommand{\stack}[2]{\ensuremath{\genfrac{}{}{0pt}{}{#1}{#2}}} 

\newtheoremstyle{exercise}
  {3pt}
  {6pt}
  {}
  {}
  {\bfseries}
  {:}
  { }
   {}
\theoremstyle{exercise}
\newtheorem{exercise}[theorem]{Exercise}
\newtheoremstyle{exercises}
  {3pt}
  {6pt}
  {}
  {}
  {\bfseries}
  {:}
  {\newline}
   {}

\theoremstyle{exercise}
\newtheorem{exercises}[theorem]{Exercises}

\input epsf
\def\boxit#1{\vbox{\hrule height1pt\hbox{\vrule width1pt\kern3pt
  \vbox{\kern3pt#1\kern3pt}\kern3pt\vrule width1pt}\hrule height1pt}}

\def\BC{\mathbb C}\def\BS{\mathbb S}
\def\BR{\mathbb R}\def\BH{\mathbb H}
\def\BP{\mathbb P}\def\BG{\mathbb G}

\def\hd{,...,}
\def\ww{\wedge}

\def\inv{{}^{-1}}

\def\CC{\mathbb C}

\def\ZZ{\mathbb Z}

\def\11{\mathbf 1}

\def\QQ{\mathbb Q}

\def\l{\lambda}
\def\a{\alpha}

\def\o{\omega}

\def\g{\gamma}
\def\s{\sigma}

\def\d{\delta}

\def\ot{{\mathord{ \otimes } }}
\def\op{{\mathord{\,\oplus }\,}}

\def\lra{{\mathord{\;\longrightarrow\;}}}
\def\ra{{\mathord{\;\rightarrow\;}}}

\def\dim{{\rm dim}\;}
\def\La#1{\Lambda^{#1}}

\def\op{\oplus}
\def\BH{\Bbb H}\def\BZ{\Bbb Z}

\def\ep{\epsilon}
\def\op{\oplus}

\def\s{\sigma}
\def\t{\tau}

\def\a{\alpha}

\def\g{\gamma}
\def\l{\lambda}

\def\FS{\mathfrak  S}

\def\ol{\overline}

\def\BP{\mathbb  P}
\def\BC{\mathbb  C}

\def\tcodim{\text{codim}}
\def\BR{\mathbb  R}\def\BS{\mathbb  S}

\def\ep{\epsilon}

\def\hd{, \hdots ,}

\def\inv{{}^{-1}}

\def\La#1{\Lambda^{#1}}

\def\ra{\rightarrow}

\def\tdet{\operatorname{det}}

\def\tperm{\operatorname{perm}}

\def\tmod{\operatorname{mod}}

\def\tmin{\operatorname{min}}

\def\trank{\operatorname{rank}}

\def\ww{\wedge}

\def\be{\begin{equation}}
\def\ene{\end{equation}}

\def\tsgn{{\rm{sgn}}}

\def\p{{\bold P}}
\def\np{{\bold N\bold P}}

\newcommand{\End}{\operatorname{End}}

\newcommand{\Id}{\operatorname{Id}}

\newcommand{\Hom}{\operatorname{Hom}}

\newcommand\Altgr{{\mathfrak{A}}}
\newcommand\Symgr{{\mathfrak{S}}}
\newcommand\longto\longrightarrow
\newcommand\Mcc{{\mathcal M}}
\newcommand\Bc{{\mathcal B}}
\newcommand\Dc{{\mathcal D}}
\newcommand\tDet{{\mathcal Det}}

\newcommand{\GL}{\operatorname{GL}}
\newcommand{\transp}{\operatorname{transp}}
\newcommand{\dc}{\operatorname{dc}}\newcommand{\sdc}{\operatorname{edc}}
\newcommand{\rdc}{\operatorname{rdc}}\newcommand{\srdc}{\operatorname{erdc}}
\newcommand{\perm}{\operatorname{perm}}
\newcommand{\Wt}{\operatorname{Wt}}
\newcommand{\SL}{\operatorname{SL}}

\def\trank{{\mathrm {rank}}}

\DeclareMathOperator{\rk}{rk}
\def\cM{\mathcal M}
\begin{document}

\title[permanent v. determinant with symmetry]{Permanent v. determinant: an exponential lower bound assuming
symmetry and a potential path towards Valiant's conjecture} 
\author{J.M.~Landsberg and Nicolas Ressayre}
 \begin{abstract} We initiate a study of determinantal representations with symmetry. We show that Grenet's determinantal representation  for the permanent is optimal among determinantal
 representations respecting  left multiplication by permutation
 and diagonal matrices (roughly half the symmetry group of the
 permanent).
 In particular,   if   any optimal determinantal
 representation of the permanent must be  polynomially related to one with such symmetry, then Valiant's conjecture
 on   permanent v. determinant is true. 
\end{abstract}
\thanks{Landsberg  supported by NSF grant DMS-1405348.
Ressayre   supported by ANR Project
(ANR-13-BS02-0001-01) and by Institut Universitaire de France}
\email{jml@math.tamu.edu, ressayre@math.univ-lyon1.fr}
\maketitle

\section{Introduction}

Perhaps the most studied polynomial of all is the determinant:
\be 
\tdet_n(x):=\sum_{\s\in \FS_n} \tsgn(\s) x^1_{\s(1)}x^2_{\s(2)}
\cdots\, x^n_{\s(n)},
\ene 
a homogeneous polynomial of degree $n$ in $n^2$ variables.
Here $\FS_n$ denotes the group of permutations on $n$ elements and $\tsgn(\s)$ denotes the sign of the permutation $\s$.

Despite its formula with $n!$ terms, $\tdet_n$ can be evaluated quickly, e.g., using
Gaussian elimination,  which exploits the large symmetry group of the determinant, e.g., 
$\tdet_n(x)=\tdet_n(AxB\inv)$ for any $n\times n$ matrices $A,B$ with determinant equal to one.

We will work exclusively
over the complex numbers and with homogeneous polynomials, the latter restriction only for convenience.
L. Valiant showed that  given  a homogeneous polynomial $P(y)$ in $M$
variables, there exists an $n$ and an 
affine linear map $\tilde A: \BC^M\ra \BC^{n^2}$ such that
$P=\det_n\circ \tilde A$. 
Such $\tilde A$ is called a {\it determinantal representation}
of $P$.
When $M=m^2$ and  $P$ is the 
 permanent polynomial 
\be 
\tperm_m(y):=\sum_{\s\in \FS_m}  y^1_{\s(1)}y^2_{\s(2)}
\cdots\ y^m_{\s(m)},
\ene 
L.~Valiant showed that one can take $n=O(2^m)$. 
As an algebraic analog of the $\p\neq\np$ conjecture, he also conjectured that one cannot do much better:

\begin{conjecture} [Valiant \cite{vali:79-3}]\label{valperdet}  Let $n(m)$ be a function of $m$ such that there exist     affine linear maps
$\tilde A_m: \BC^{m^2}\ra \BC^{n(m)^2}$ satisfying
\be\label{permrep}
\tperm_m=\tdet_{n(m)}\circ \tilde A_m.
\ene
 Then $n(m)$ grows faster than any polynomial.
\end{conjecture}

To measure progress towards Conjecture~\ref{valperdet},  define
$\dc(\tperm_m)$ to be the smallest $n(m)$ such that there exists $\tilde A_m$ satisfying  \eqref{permrep}.  The conjecture
is that $\dc(\tperm_m)$ grows faster than any polynomial in $m$. 
Lower bounds on $\dc(\tperm_m)$ are:  $\dc(\tperm_m)>m$ (Marcus and Minc  \cite{MR0147488}), $\dc(\tperm_m)>1.06m$ (Von zur Gathen   \cite{MR910987}),
 $\dc(\tperm_m)>\sqrt{2}m-O(\sqrt{m})$ (Meshulam, reported in \cite{MR910987}, and Cai \cite{MR1032157}), with the current world record $\dc(\tperm_m)\geq \frac{m^2}2$
 \cite{MR2126826} by Mignon and the second author. (Over $\BR$, Yabe recently showed that $\dc_{\BR}(\tperm_m)\geq m^2-2m+2$ \cite{DBLP:journals/corr/Yabe15}, and
 in \cite{journals/cc/CaiCL10}
 Cai, Chen and Li  extended the $\frac{m^2}2$ bound to arbitrary fields.)

 \medskip

Inspired by {\it Geometric Complexity Theory} (GCT) \cite{MS1}, we focus on the {\it symmetries} of $\tdet_n$ and $\tperm_m$.
Let $V$ be a complex vector space of dimension $M$, let
$\GL(V)$ denote the group of invertible linear maps $V\ra V$. For
$P\in S^mV^*$, a homogeneous polynomial of degree $m$ on $V$, let
\begin{align*}
 G_P:&=\{ g\in \GL(V)\mid P(g\inv y)=  P(y)\quad \forall y\in V\}\\
\BG_P:&=\{ g\in \GL(V)\mid P(g\inv y)\in \BC^* P(y) \quad \forall y\in V\}
\end{align*}
denote the {\it   symmetry group (resp. projective symmetry group) }
of $P$. The function $\chi_P: \BG_P\ra \BC^*$ defined by the equality 
$P(g\inv y)=\chi_P(g)P(y) $ is group homomorphism called
the {\it character} of $P$.
For example $\BG_{\tdet_n}\simeq(\GL_n\times \GL_n)/\BC^*\rtimes
\BZ_2$ \cite{Frobdet}, where the $\GL_n\times \GL_n$ invariance comes
from $\det(AxB^{-1})=(\det A\det B\inv)\det(x)$ and the $\BZ_2$ is because $\tdet_n(x)=\tdet_n(x^T)$ where $x^T$ is the transpose of the matrix $x$.
Write $\t: \GL_n\times \GL_n\ra \GL_{n^2}$ for the map $(A,B)\mapsto \{ x\mapsto AxB^{-1}\}$. The character
$\chi_{\tdet_n}$ satisfies $\chi_{\tdet_n}\circ\tau(A,B)=\det(A)\det(B)^{-1}$.

As observed in \cite{MS1}, the permanent (resp. determinant) is {\it characterized
by its symmetries} 
in the sense that any polynomial $P\in  S^m\BC^{m^2*}$
 with a symmetry group $G_P$ such that $G_P\supseteq G_{\tperm_m}$
 (resp. $G_P\supseteq G_{\tdet_m}$)
is a scalar multiple of the permanent (resp. determinant). This property
is the cornerstone of GCT. The program outlined
in \cite{MS1,MS2} is an approach to Valiant's conjecture based on  
  the   functions on $\GL_{n^2}$ that respect the symmetry group
$G_{\tdet_n}$, i.e.,  are invariant under the action of $G_{\tdet_n}$.

The interest in considering $\BG_P$ instead of $G_P$ is that if $P$ is
characterized by $G_P$ among homogeneous polynomials of the same 
degree, then it is
characterized by the pair $(\BG_P,\chi_P)$ among all polynomials.  
This will be useful, since {\it a priori}, $\tdet_n\circ\tilde A$ need not be
homogeneous. 

Guided by the    principles of GCT, we ask: 

\begin{center}
  {\it What are the $\tilde A$ that respect the
symmetry group of the permanent?}
\end{center}

\medskip
To make this question   precise, let $\Mcc_n(\BC)$ denote the space of $n\times n$ matrices
and write $\tilde A=\Lambda + A$ where $\Lambda\in \Mcc_n(\BC)$ is
a fixed matrix and $A: \Mcc_m(\BC)\ra  \Mcc_n(\BC)$ is linear.

The subgroup $\BG_{\tdet_n}\subset \GL_{n^2}$
satisfies
\be\label{gdetneqn}
\BG_{\tdet_n}\simeq(\GL_n\times \GL_n)/\BC^*\rtimes \BZ_2.
\ene

\begin{definition}\label{detrepdef}
Let $\tilde A\,:\,V\longto \Mcc_n(\BC)$ be a determinantal representation of $P\in S^mV^*$.
Define 
$$\BG_{  A}=\{ g\in \BG_{\tdet_n}\mid g\cdot \Lambda=\Lambda {\rm\ and\ } g\cdot A(V)=A(V)\},
$$
  the {\it symmetry group of the determinantal
representation}  $\tilde A$ of $P$.
\end{definition}

The group $\BG_{  A}$ 
comes with a representation
$\rho_A\,:\,\BG_A\longto \GL(A(V))$ obtained by
restricting the action to $A(V)$. 
We assume that $P$ cannot be expressed using   $\dim(V)-1$ variables, i.e., that $P\not\in S^mV'$ for any
hyperplane $V'\subset V^*$. Then $A\,:\,V\longto A(V)$ is bijective. 
Let  $A^{-1}\,:\,A(V)\longto V$ denote its inverse. 
Set 
\begin{align}\label{rhoa}
\bar\rho_A\,:\,\BG_A&\longto \GL(V)\\
\nonumber g&\longmapsto
A\circ\rho_A(g)\circ A^{-1}.
\end{align}


\begin{definition} We say $\tilde A$ {\it respects the symmetries} of $P$ if \eqref{rhoa} is surjective.
We also refer to such $\tilde A$ as a {\it equivariant representation}
of $P$.

If $G$ is a subgroup of $\BG_P$, we say that {\it $\tilde A$ respects $G$}
if $G$ is contained in the image of $\bar\rho_A$.
\end{definition}

\begin{example}\label{Qex} Let $Q=\sum_{j=1}^{M} z_j^2\in S^2\BC^{M*}$ be a nondegenerate quadric. 
Then $\BG_Q=\BC^*\times O(M)$ where $O(M)=\{ B\in GL_{M}\mid B\inv=B^T\}$ is the orthogonal
group, as for such $B$, $B\cdot Q=\sum_{i,j,k}B_{i,j} B_{k,j}z_iz_k=\sum_{ij}\d_{ij}z_iz_j=Q$.
Consider the determinantal representation  
\be\label{quadrespect}
Q=\tdet_{M+1} \begin{pmatrix} 0 & -z_1& \cdots & -z_{M}\\
z_1 & 1 & & \\
\vdots & & \ddots &   \\
z_{M} & & & 1
\end{pmatrix}.
\ene
For  $(\l,B)\in \BG_Q$,  define an action on $Z\in \cM_{M+1}(\BC)$ by
$$
Z\mapsto
\begin{pmatrix} \lambda&0\\
0&B
\end{pmatrix} 
Z
\begin{pmatrix}
  \lambda^{-1}&0\\
0&B
\end{pmatrix}^{-1}.
$$
Write
$$
X=
\begin{pmatrix}
  x_1\\ \vdots\\x_{M}
\end{pmatrix}\qquad{\rm so}\qquad
\tilde A=
\begin{pmatrix}
  0&-X^T\\
X&\Id_{M}
\end{pmatrix}.
$$
The relation $B^{-1}=B^T$ implies 
$$
\begin{pmatrix}
  \lambda&0\\
0&B
\end{pmatrix}\cdot
\begin{pmatrix}
  0&-X^T\\
X&\Id_{M}
\end{pmatrix}
\cdot
\begin{pmatrix}
  \lambda^{-1}&0\\
0&B
\end{pmatrix}^{-1}
=\begin{pmatrix}
  0&-(\lambda BX)^T\\
\lambda BX&\Id_{M}
\end{pmatrix}.
$$
Taking $\tdet_{M+1}$ on both sides gives   
$$
\lambda^2Q(X)=(\lambda,B)\cdot Q(   X).
$$
Thus   $\tilde A$ respects  the symmetries of $Q$.
\end{example}

Roughly speaking, $\tilde A$ respects the symmetries of $P$ if for any
$g\in \BG_P$, one can recover the fact that   $P(g\inv x)=\chi_P(g)P(x)$ just by
applying Gaussian elimination and $\det(Z^T)=\det(Z)$ to $\tilde A$.

\begin{definition}
For $P\in S^mV^*$, define the {\it equivariant determinantal complexity} of $P$, denoted  $\sdc(P)$, to be   the smallest $n$ such that
there is an equivariant determinantal representation 
of $P$.
\end{definition}

Of course $\sdc(P)\geq \dc(P)$.
We do not know if $\sdc(P)$ is finite in general. Our main result
is  that  $\sdc(\tperm_m)$ is exponential in $m$.  

\section{Results}

\subsection{Main Theorem} 

\begin{theorem}\label{mainthm} Let $m\geq 3$. Then
   $\sdc(\tperm_m)= \binom{2m}{m}-1 \sim  4^m$. 
\end{theorem}

There are several instances in complexity theory where an
optimal algorithm partially respects symmetry, e.g. Strassen's algorithm for $2\times 2$ matrix multiplication
respects the $\BZ_3$-symmetry of the matrix multiplication operator (see \cite[\S 4.2]{Lsimons}), but not
the $GL_2^{\times 3}$ symmetry.

For the purposes of Valiant's conjecture,
we ask the weaker question:

\begin{question}\label{symq} Does there exist a polynomial $e(d)$ such that
  $\sdc(\tperm_m)\leq e(\dc(\tperm_m))$?
\end{question}

Theorem~\ref{mainthm} implies:

\begin{corollary} If the answer to Question~\ref{symq} is affirmative, then
Conjecture~\ref{valperdet} is true.
\end{corollary}

We have no opinion as to  what  the answer to Question~\ref{symq} should be, but as it provides a 
new potential path to proving Valiant's conjecture, it merits
further investigation.  
Note that Question~\ref{symq} is a {\it flip}  in the  
terminology  of \cite{Mflip},  since a positive answer is an  
existence result. 
It fits into the  more general question:  {\it  When an object
has symmetry, does it admit an  optimal expression that preserves its symmetry?}

\begin{example}  
Let $T\in W^{\ot d}$ be a symmetric tensor, i.e. $T\in S^dW\subset W^{\ot d}$.
Say  $T$ can be written as a sum of $r$ rank one tensors, then  P. Comon  conjectures \cite{Como02:oxford}  that it can be written
as a sum of $r$ rank one symmetric tensors. 
\end{example}

\begin{example} [Optimal Waring decompositions]\label{optimalwaring}
The optimal Waring decomposition of $x_1\dots\, x_n$, dating back at least  to
\cite{MR1573008} and proved to be optimal in \cite{MR2842085} is
\be\label{fischer}
 x_1\dots\, x_n= \frac{1}{2^{n-1} n!} \sum_{\stack{\epsilon \in \{-1,1\}^{n}}{\ep_1=1}}\big(
 \sum_{j=1}^n\ep_jx_j\big)^n
\Pi_{i=1}^n\ep_i   ,
\ene
a sum with $2^{n-1}$ terms. This decomposition has an $\FS_{n-1}$-symmetry but
not an $\FS_n$-symmetry, nor does it preserve the action of the torus $T^{SL_n}$ of diagonal matrices
with determinant one. One can obtain an $\FS_n$-invariant expression by doubling the size:
\be\label{fischerdoub}
 x_1\dots\, x_n= \frac{1}{2^{n } n!} \sum_{ \epsilon \in \{-1,1\}^{n} }\big(
 \sum_{j=1}^n\ep_jx_j\big)^n
\Pi_{i=1}^n\ep_i   ,
,
\ene
because
\begin{align*}
(- x_1 + \epsilon_2 x_2  + \dots +& \epsilon_{n } x_n)^n (-1)\epsilon_2 \dots\, \epsilon_{n }\\
=&(-1)^n(  x_1 + (-\epsilon_2) x_2 +  \dots + (-\epsilon_{n }) x_n)^n (-1)\epsilon_2 \dots\, \epsilon_{n }\\
 =& (  x_1 + (-\epsilon_2) x_2 +  \dots + (-\epsilon_{n }) x_n)^n (-\epsilon_2) \dots\,(- \epsilon_{n }).
\end{align*}

The optimal Waring decomposition of the permanent is not known, but it
is known to be of size greater than ${\binom n{\lfloor n/2\rfloor}}^2\sim 4^n/\sqrt{n}$.
The Ryser-Glynn formula \cite{MR2673027} is
\be\label{ryserformula}
  \tperm_n(x) = 2^{-n+1} \sum_{\substack{\epsilon \in \{-1,1\}^n \\ \epsilon_1=1}}
    \prod_{1 \leq i \leq n} \sum_{1 \leq j \leq n} \epsilon_i \epsilon_j x_{i,j} ,
\ene
the outer sum taken over $n$-tuples $(\epsilon_1=1, \epsilon_2,\dots,\epsilon_n)$.
This $\FS_{n-1}\times \FS_{n-1}$-invariant formula can also be made $\FS_n\times\FS_n$-invariant
by enlarging  it by a factor of $4$, to get a $\FS_n\times \FS_n$ homogeneous depth three formula
that is within a factor of four of the best known. Then expanding each monomial above, using Equation  \eqref{fischerdoub},
one gets a $\FS_n\times \FS_n$-Waring expression within a factor of $O(\sqrt{n})$ of the lower bound.
\end{example}

\begin{example}\label{aviex}Examples regarding equivariant representations of $\FS_N$-invariant functions from the Boolean world give
inconclusive indications regarding  Question \ref{symq}.

The $MOD_m$-degree of a Boolean function
$f(x_1\hd x_N)$ is the smallest degree of any polynomial $P\in \BZ[x_1\hd x_N]$
such that $f(x)=0$ if and only if $P(x)=0$ for all $x\in \{ 0,1\}^N$. 
The known upper bound for the $MOD_m$-degree of the Boolean  $OR$ function ($OR(x_1\hd x_N)=1$ if any $x_j=1$ and is zero
if all $x_j=0$) is attained by symmetric polynomials \cite{MR1313536}. Moreover in \cite{MR1313536} it is also shown that
this bound cannot be improved with symmetric polynomials, and it is far from the known lower bound.

The boolean majority  function $MAJ(x_1\hd x_N)$ takes on $1$ if at least half the $x_j=1$ and zero otherwise.
The best monotone Boolean formula for $MAJ$ \cite{MR756162}  is polynomial in $N$ and  attained using random functions, and it is expected that
the only symmetric monotone
formula for majority is the trivial one, disjunction of all $\frac n2$-size
subsets (or its dual), which is of exponential size.

\end{example}

\begin{question} Does every $P$ that is determined by its 
symmetry group admit a determinantal representation that
respects its symmetries? For those $P$ that do, how much
larger must such a determinantal representation be from 
the size of a minimal one?
\end{question}

\subsection{Grenet's formulas}\label{grenetsect}
The starting point of our investigations was the result in \cite{2015arXiv150502205A} 
that $\dc(\tperm_3)=7$, in particular 
  Grenet's representatation  \cite{Gre11} for $\tperm_3$:
  \be\label{grenet3}
\tperm_3(y)=\tdet_7  
\begin{pmatrix}
0&0&0&0& y_3^3&  y_2^3& y_1^3\\
y^1_1 & 1&  &    & & &  \\
y^1_2 &  & 1 &    & & &  \\
y^1_3 &  &   & 1   & & &  \\
& y^2_2 &  y^2_1 &0& 1& & \\
& y^2_3 & 0 & y^2_1&  & 1& \\
&0  & y^2_3 & y^2_2&  &  & 1
\end{pmatrix},
\ene
 is optimal. We sought  to
 understand \eqref{grenet3} from a geometric perspective.  A first observation is that
 it, and more generally Grenet's representation
 for $\tperm_m$ as a determinant of size $2^m-1$ respects about half the symmetries of
 the permanent.  In particular,   the optimal expression for $\tperm_3$ respects
 about half its symmetries.
 
 

\medskip

To explain this observation, introduce
the following  notation. Write $\Mcc_m(\BC)=\Hom(F,E)=F^*\ot E$, where
$E,F=\BC^m$.
This distinction of the two copies of $\BC^m$  clarifies 
the action of the group $\GL( E)\times\GL(F)$ on $\Hom(F,E) $. This action is  
$(A,B).x=AxB^{-1}$, for any $x\in \Hom(F,E)$ and $(A,B)\in\GL( E)\times\GL(F)$. 
Let $T^{\GL(E)}\subset \GL(E)$ consist of the diagonal matrices 
and  let  $ N(T^{\GL(E)})= T^{\GL(E)}\rtimes \Symgr_m\subset \GL(E)$    be its normalizer, where
$\Symgr_m$ denotes the group of permutations on $m$
elements. Similarly for  $T^{\GL(F)}$ and $N(T^{\GL(F)})$.
Then  
$\BG_{\tperm_m}\simeq  [(N(T^{\GL(E)})\times N(T^{\GL(F)}))/\CC^*]\rtimes \BZ_2$, where
the embedding of $(N(T^{\GL(E)})\times N(T^{\GL(F)}))/\CC^*$ in
$\GL(\Hom(F,E))$ is given by the action above and  the term $\BZ_2$
corresponds to transposition. 

The following refinement  of Theorem~\ref{mainthm} asserts that to get
an exponential lower bound it is sufficient   to respect about half the symmetries of
 the permanent. 

\medskip
 \begin{theorem}
   \label{halfsdcperm}
Let $m\geq 3$. Let $\tilde A_m\, :\, \cM_m(\BC)\lra \cM_n(\BC)$
 be a determinantal representation of $\tperm_m$ that respects
 $N(T^{\GL(E)})$. Then $n\geq 2^m-1$.

Moreover, Grenet's determinantal representation of $\tperm_m$ respects
$N(T^{\GL(E)})$ and has size $2^m-1$.
 \end{theorem}

\medskip
We now   explain Grenet's expressions  from a representation-theoretic  perspective.
Let $[m]:=\{ 1\hd m\}$ and let  $k\in[ m ]$. Note that  $S^kE$ is an   irreducible $\GL(E)$-module but it is
is not irreducible as an $N(T^{\GL(E)})$-module.
For example, let $e_1\hd e_m$ be a basis of $E$,  and let $(S^kE)_{reg}$ denote the span of $\prod_{i\in I} e_i$, for $I\subset [m]$ of
cardinality $k$ (the space spanned by the square-free monomials, also known as the space of {\it regular} weights): $(S^kE)_{reg}$  is an irreducible
$N(T^{\GL(E)})$-submodule of $S^kE$. 
 Moreover, there exists a unique $N(T^{\GL(E)})$-equivariant
 projection $\pi_k$ from $S^kE$ to $(S^kE)_{reg}$.

For $v\in E$, define  $s_k(v): (S^kE)_{reg}  
 \ra (S^{k+1}E)_{reg}$ to be 
multiplication by $v$ followed by $\pi_{k+1}$.
Alternatively, $(S^{k+1}E)_{reg}$ is an $ N(T^{\GL(E)})$-submodule of $E\ot (S^{k}E)_{reg}$, and
$s_k: E\ra   (S^{k}E)_{reg}^*\ot  (S^{k+1}E)_{reg}$ is the unique $ N(T^{\GL(E)})$-equivariant inclusion.
Let $\Id_W: W\ra W$ denote the identity map on the vector space $W$. Fix a basis $f_1\hd f_m$ of $F^*$. 

\medskip
\begin{proposition}\label{grenetrep}
The following is Grenet's determinantal representation of $\tperm_m$.
Let $\BC^n=\bigoplus_{k=0}^{m-1}(S^kE)_{reg}$, so $n =2^m-1$,
 and identify $S^0E\simeq (S^mE)_{reg}$.
Set $$\Lambda_0= \sum_{k=1}^{m-1}\Id_{(S^k E)_{reg}}
$$ and  define
\be
\label{Grenetexpr}\tilde A=\Lambda_0 + \sum_{k=0}^{m-1} s_k\ot f_{k+1}.
\ene
Then $(-1)^{m+1}\tperm_m= \tdet_n\circ \tilde A$. To obtain the permanent exactly, replace
$\Id_{(S^1E)_{reg}}$ by $(-1)^{m+1}\Id_{(S^1E)_{reg}}$ in the formula for $\Lambda_0$.

In bases respecting the block decomposition induced from the direct sum, the linear part, other than the last term which lies
in the upper right block,   lies just below the
diagonal blocks, and all blocks other than  the upper right block  and  the diagonal and sub-diagonal blocks,
are zero. 

Moreover  $ N(T^{\GL(E)})\subseteq \bar\rho_A(\BG_A)$.
\end{proposition}

\subsection{An equivariant representation of the permanent}

We   now  give a minimal equivariant determinantal representation
of $\tperm_m$. By Theorem~\ref{mainthm},  its size is $\binom{2m}{m}-1$.
For $e\ot f\in E\ot F^*$, let  
$S_k(e\ot f): (S^kE)_{reg}  \ot (S^kF^*)_{reg} 
 \ra (S^{k+1}E)_{reg}\ot (S^{k+1}F^*)_{reg}$ be 
multiplication by $e$ on the first factor and $f$ on the second followed by projection into $(S^{k+1}E)_{reg}\ot (S^{k+1}F^*)_{reg}$.
Equivalently,
$$
S_k:(E\ot F^*)\ra  ((S^kE)_{reg}  \ot (S^kF^*)_{reg} )^*\ot   (S^{k+1}E)_{reg}\ot (S^{k+1}F^*)_{reg}
$$ 
is the unique $  N(T^{\SL(E)})\times   N(T^{\SL(F )})$
equivariant inclusion.

\medskip
\begin{proposition}\label{permwithrespect}
The following is an equivariant determinantal representation of $\tperm_m$:
Let $\BC^n=\oplus_{k=0}^{m-1}(S^kE)_{reg}\ot (S^kF^*)_{reg}$, so
$ n=\binom{2m}{m}-1 \sim  4^m$. Fix a linear isomorphism  $  S^0E\ot S^0F^*\simeq (S^mE)_{reg}\ot(S^mF^*)_{reg} $. 
Set 
$$\Lambda_0=   
\sum_{k=1}^{m-1}\Id_{(S^k E)_{reg}\ot (S^k F^*)_{reg}}
$$ 
and  define
\be
\label{symGrenetexpr}\tilde A=  (m!)^{\frac {-1}{n-m}}\Lambda_0  +
\sum_{k=0}^{m-1} S_k.
\ene
Then  $(-1)^{m+1}\tperm_m= \tdet_n\circ \tilde A$. 
In bases respecting the block structure induced by the direct sum, except for $S_{m-1}$, which lies in the 
upper right hand block, the linear part   lies just below the
diagonal block. 
\end{proposition}

\subsection{Determinantal representations of quadrics}\label{quadricsres}
It will be instructive to examine other polynomials determined by their symmetry groups. Perhaps
the simplest such is a nondegenerate  quadratic form. 

Let
$Q=\sum_{j=1}^s x_jy_j$ be a non-degenerate  quadratic
form in $2s$ variables. The polynomial $Q$ is characterized by its symmetries.
By \cite{MR2126826}, if $s\geq 3$, the smallest  determinantal representation
of $Q$ is of size $s+1$:
\be\label{spone}
\tilde A= \begin{pmatrix}
0& -x_1& \cdots & -x_s\\ 
y^1 & 1 & & \\
\vdots & &\ddots   &  \\
y^s & & & 1\end{pmatrix}.
\ene
As described in \S\ref{sec:quad}, this representation respects about \lq\lq half\rq\rq\ the symmetry group $\BG_Q$.
We show  in  \S\ref{sec:quad}  that there is no size $s+1$ determinantal representation respecting $\BG_Q$. However, 
Example \ref{Qex} shows there is a size $2s+1$ determinantal representation respecting $\BG_Q$.


\begin{proposition}\label{quadprop}
  Let $Q\in S^2\BC^{M*}$ be a nondegenerate  quadratic form, that is, a   homogeneous polynomial
   of degree 2. Then
$$
\sdc(Q)=M+1.
$$
\end{proposition}

\subsection{Determinantal representations of the determinant}
Although it may appear strange at first, one can ask for determinantal representations of $\det_m$. In this case, to get an interesting lower
bound, we add a regularity condition motivated by Lemma \ref{lem:vzg}:

\begin{definition}
  \label{def:reg}
Let $P\in S^mV^*$. A determinantal representation $\tilde A\,:\,V\longto
\Mcc_n(\BC)$ is said to be {\it regular} if $\tilde A(0)$ has rank
$n-1$.

Call the minimal size of  a regular determinantal representation of $P$
 the {\it regular determinantal complexity of $P$} and denote it by
$\rdc(P)$. Let  $\srdc(P)$ denote the minimal size of a regular equivariant determinantal representation of $P$.
\end{definition}

 Any determinantal representation of  $\tperm_m$ or a
 smooth quadric  is regular, see \S\ref{normallambda}. 
In contrast,  the trivial   determinantal representation of  $\tdet_m$
is not regular; but this representation is equivariant so $\sdc(\tdet_m)=m$.

\begin{theorem}
  \label{th:srdcdet}
$\srdc(\tdet_m)= \binom{2m}{m}-1 \sim  4^m$. 
\end{theorem}

\medskip As in the case of the permanent, we can  get
an exponential lower bound using only  about half the symmetries of
 the determinant. 

\medskip
 \begin{theorem}
   \label{halfsdcdet}
  Let $\tilde A_m\, :\, \cM_m(\BC)\lra \cM_n(\BC)$
 be a regular determinantal representation of $\tdet_m$ that respects
 $\GL(E)$. Then $n\geq 2^m-1$.

Moreover, there exists a  regular determinantal representation of $\tdet_m$ that
respects $\GL(E)$ of size $2^m-1$.
 \end{theorem}
 
 \begin{remark} Normally when one obtains the same lower bound for
 the determinant as the permanent in some model it is discouraging for the model.
 However here there is an important difference due to the imposition of
 regularity for the determinant. We discuss this in further  below Question \ref{rdcdet}.
 \end{remark}

\medskip
We now introduce   notation to describe the regular determinantal representation of $\tdet_m$ that
respects $\GL(E)$ of size $2^m-1$ mentioned in Theorem~\ref{halfsdcdet}.

Observe that $(S^kE)_{reg}$ is isomorphic to $\Lambda^kE$ as a
$T^{\GL(E)}$-module but not as an $\FS_m$-module. The irreducible $\FS_m$-modules are
indexed by partitions of $m$, write $[\pi]$ for the $\FS_m$-module associated to the partition $\pi$. 
Then as $\FS_m$-modules, $(S^kE)_{reg}=\oplus_{j=0}^{\tmin\{k,m-k\} } [m-j,j]$, while
$\La kE=[m-k,1^k]\op [m-k+1,1^{k-1}]$. In particular these spaces are not isomorphic as $N(T^{SL(E)})$-modules.

Write $\Mcc_m(\BC)=E \ot F^*$. Let $f_1\hd f_m$ be a basis of $F^*$.
Let $ex_k$ denote exterior multiplication in $E$:
\begin{align*} ex_k\,:\,E&\longto (\La kE)^*\ot (\La{k+1}E)\\
v &\mapsto \{ \o\mapsto v\ww \o\}.
\end{align*}

\begin{proposition}\label{dethalfrep}
The following is a   regular determinantal representation of $\tdet_m$
that respects $\GL(E)$.  
Let $\BC^n=\bigoplus_{j=0} ^{m-1}\La j E$,   so $n=2^m-1$ and 
  $\End(\BC^n)=\oplus_{0\leq i,j\leq m-1}\Hom(\La j E,\La i E)$. 
Fix an identification $\La mE\simeq \La 0 E$.
Set 
$$\Lambda_0= \sum_{k=1}^{m-1}\Id_{\La k E},
$$
and
\be\label{regdet}\tilde A=\Lambda_0 + \sum_{k=0}^{m-1} ex_k\ot f_{k+1}.
\ene
Then $\tdet_m=  \tdet_n\circ \tilde A$ if $m\equiv 1,2\tmod 4$ and $\tdet_m= -\tdet_n\circ \tilde A$ if
$m\equiv 0,3\tmod 4$.  
In bases respecting the direct sum, 
the linear part, other than the last term which lies in the upper right block,    lies just below the
diagonal blocks, and all blocks other than the upper right,  the diagonal and sub-diagonal
are zero. 
\end{proposition}

Note the similarity with the expression \eqref{Grenetexpr}. This will be useful
for proving the results about the determinantal representations of the permanent.

\bigskip
When $m=2$ this is
$$
\begin{pmatrix}
0& -y^2_2 &  y^2_1\\
y^1_1 & 1& 0\\
y^1_2 & 0 & 1
\end{pmatrix}
$$
agreeing with our earlier calculation of a rank four quadric. Note the minus sign in front of $y^2_2$ because
$ex(e_2)(e_1)=-e_1\ww e_2$.

For example, ordering the bases of $\La 2\BC^3$ by $e_1\ww e_2,
e_1\ww e_3, e_2\ww e_3$, the matrix for $\tdet_3$ is
$$
\begin{pmatrix}
0&0&0&0& y_3^3& -y_2^3& y_1^3\\
y^1_1 & 1&  &    & & &  \\
y^1_2 &  & 1 &    & & &  \\
y^1_3 &  &   & 1   & & &  \\
& -y^2_2 &  y^2_1 &0& 1& & \\
& -y^2_3 & 0 & y^2_1&  & 1& \\
&0  & -y^2_3 & y^2_2&  &  & 1
\end{pmatrix}.
$$

 \medskip
We   now   give a   regular equivariant determinantal representation
of $\tdet_m$.
Let $EX_k$ denote the exterior multiplication
\begin{align*} EX_k\,:\,E\ot F^*&\longto (\La kE\ot \La k F^*)^*\ot (\La{k+1}E\ot \La{k+1}F^*)\\
e\ot f &\mapsto \{ \o \ot \eta \mapsto  e\ww \o\ot f\ww \eta \},
\end{align*}

\begin{proposition}\label{detwithrespect}
The following is an  equivariant   regular determinantal representation of $\tdet_m$.
Let $\BC^n=\bigoplus_{j=0}^{m-1}\La j E\ot \La j F^*$,   so $n=\binom{2m}{m}-1 \sim  4^m $   
and  $\End(\BC^n)=\oplus_{0\leq i,j\leq m}\Hom(\La j E\ot \La j F^*,\La i E\ot \La i F^*)$. 
Fix an identification $ \La mE\ot\La mF^*\simeq \La 0 E\ot \La 0 F^*$.
Set $$\Lambda_0=  \sum_{k=1}^{m-1}\Id_{\La k E\ot \La k F^*} 
$$ and 
define
\be
\label{symdetexpr}\tilde A=(m!)^{\frac {-1}{n-m}} \Lambda_0 +
\sum_{k=0}^{m-1} EX_k. 
\ene
Then $(-1)^{m+1}\tdet_m=\tdet_n\circ \tilde A$. 
\end{proposition}

Comparing Theorems~\ref{mainthm} and \ref{th:srdcdet},
Theorems~\ref{halfsdcperm} and \ref{halfsdcdet},
Propositions~\ref{grenetrep} and \ref{dethalfrep} and Propositions~\ref{permwithrespect} and
\ref{detwithrespect}, one can see that  $\tdet_m$ and $\tperm_m$ have the
same behavior relatively to equivariant regular determinantal
representations.
This prompts the   question:
What is the regular determinantal complexity of the determinant? In particular:

\begin{question}
\label{rdcdet}
   Let $\rdc(\tdet_m)$ be the smallest value of $n$  such that there exist    affine linear maps
$\tilde A_m: \BC^{m^2}\ra \BC^{n^2}$ such that
\be\label{detrep}
\tdet_m=\tdet_{n}\circ \tilde A_m\mbox{ and }\trank{\tilde A(0)}=n-1.
\ene
What is the growth of $\rdc(\tdet_m)$?
\end{question}

Because of the symmetries of $\tdet_n$, this question might
be easier than determining the growth of $\dc(\tperm_m)$. If $\rdc(\tdet_m)$ grows faster than a polynomial, this would
provide a path to proving  Valiant's conjecture by  trying to transport
the bound to the permanent
via  the Howe-Young duality functor \cite{MR2308168},  which
guided the proofs of the permanental cases in this article
and enabled the computation of the linear strand of the minimal free
resolution of the ideal generated by subpermanents in \cite{2015arXiv150405171E}.
Even if this fails, it would still   refocus research towards the large singular locus
of $\{\tdet_n=0\}$. If $\rdc(\tdet_m)$ grows polynomially, then
one could still try to transport the bound to the permanent. If one
is unable to do so,  the breakdown of the method could give better insight to the difference between
$\tdet_m$ and $\tperm_m$. Polynomial growth  would 
also give a negative answer to the analog of Question \ref{symq} for the
determinant.

\subsection{Overview}
In \S\ref{prelimsect} we establish basic facts about determinantal representations and review
results about algebraic groups. The proofs of the results are then presented in increasing order
of difficulty, beginning with the easy case of quadrics in \S\ref{sec:quad}, then the case
of the determinant  in \S\ref{sec:det}, and concluding with the permanent in \S\ref{sec:perm}.

\subsection{Acknowledgments} The seed  for this article was planted during the 
Fall 2014 semester   program {\it Algorithms and Complexity in Algebraic Geometry}  at the 
Simons Institute for the Theory of Computing, UC Berkeley. Most of the work was done when the
Landsberg was a guest of   Ressayre  and Pascal Koiran. Landsberg thanks his hosts as well as  U. Lyon and ENS Lyon for
their hospitality and support. 
We thank:  Christian Ikenmeyer for useful discussions about equivariant determinantal presentations of the permanent
and Grenet's algorithm, Josh Grochow, Christian Ikenmeyer  and Shrawan Kumar for useful suggestions for improving the exposition,  J\'er\^ome Germoni for mentioning the
existence of the spin symmetric group,  Michael Forbes for suggesting Example~\ref{optimalwaring}, and Avi Wigderson for
suggesting Example~\ref{aviex}.

\section{Preliminaries on symmetries}\label{prelimsect}

Throughout this section $P\in S^mV^*$ is a polynomial and 
$$\tilde
A=\Lambda + A : V\longto \Mcc_n(\BC)
$$ is
a determinantal representation of $P$. 

\subsection{Notation} For an affine algebraic group  $G$, 
$G^\circ$ denotes the connected component of the identity.  The 
homomorphisms from $G$ to $\CC^*$ are called {\it characters}  of $G$. 
They form an abelian group denoted by $X(G)$. The law in $X(G)$ is
denoted additively  (even if it comes from multiplication in
$\CC^*$). If $G\simeq (\CC^*)^{\times r}$ is a torus then $X(G)\simeq \ZZ^r$. 

 For a vector space  $V$, $\BP V$ is the associated
projective space.
 For  a polynomial $P$  on  $V$,   $\{P=0\}=\{x\in V\,:\,P(x)=0\}$ denotes its zero set, which is an 
affine algebraic variety. For $v\in V$, the differential of $P$ at $v$ is denoted  $d_vP\in V^*$
  and   $\{P=0\}_{sing}=\{x\in
V\,:\, P(x)=0\mbox{ and } d_xP=0\}$ denotes the
singular locus of $\{P=0\}$. Note  that we do not consider the
reduced algebraic variety, in particular if $P$ is a square $\{P=0\}_{sing}=\{P=0\}$.
 
\subsection{Representations of $\BG_A$}\label{rhoadef}
The following observation plays a key
role: 

\begin{lemma}\label{containlem}
    $\bar\rho_A(\BG_A)\subset \BG_P$. Moreover, for any $g\in \BG_A$,
$\chi_{\tdet_n}(g)=\chi_P(\bar\rho_A(g))$.
\end{lemma}

\begin{proof}
 Let $g\in \BG_A$ and $v\in V$. Then
%
\begin{align*}
(\bar\rho_A(g)P)(v)&= P(\bar\rho_A(g)^{-1}v)\\
&=\tdet_n(\Lambda + A(A^{-1}g^{-1}A(v)))\\
&=\tdet_n(g^{-1}(\Lambda + A(v))\\
&=(g\tdet_n)(\tilde A(v))\\
&= \chi_{\tdet_n}(g)\tdet_n(\tilde A(v))\\
&=\chi_{\tdet_n}(g) P(v).
\end{align*}
The lemma follows.
\end{proof}

\subsection{Normal form for $\Lambda$}\label{normallambda}

\begin{lemma}\label{lem:vzg}\cite{MR910987} Let $P\in S^mV^*$ be a polynomial.  If $\tcodim(\{P=0\}_{sing}, V)\geq 5$, then
any determinantal
representation $\tilde A$  of $P$ is regular.
\end{lemma}

\begin{proof}
Consider the affine variety $\Dc_n=\{\det_n=0\}$.
The singular locus of $\Dc_n$ is the set of 
matrices of rank at most $n-2$, and hence has codimension 4 in
$\Mcc_n(\BC)$.

For any $v\in V$, we have $d_vP=d_v(\tdet_n\circ \tilde A)=d _{\tilde A(v)}(\tdet_n)\circ A$. 
In particular, if $\tilde A(v)\in (\Dc_n)_{sing}$ then $v\in \{ P=0\}_{sing}$.

But the set of $v$ such that $\rk(\tilde A(v))\leq n-2$ is either
empty, or its codimension is at most 4. The assumption of the lemma
implies that $\rk(\tilde A(v))\geq n-1$ for any $v\in V$.
 In particular $\rk(\tilde A(0))=n-1$.
\end{proof}

\smallskip

In \cite{MR910987}, von zur Gathen showed that  $\tcodim (\{\tperm_{m}=0\}_{sing}, \BC^{m^2})\geq~5$.

 \smallskip

Let $\Lambda_{n-1} \in\Mcc_n(\CC)$ be the matrix with 1 in  the $n-1$
last diagonal entries and $0$ elsewhere. 
Any   determinantal representation $\tilde A$ of $P$ of size $n$ 
with $\trank(\tilde A(0))=n-1$  can be transformed (by
multiplying on the left and  
right by constant invertible matrices) to a determinantal representation of $P$ satisfying
$\tilde A(0)=\Lambda_{n-1} $.

\subsection{An auxiliary symmetry group}\label{auxgroup}
Recall the group $\BG_{\tdet_n}\subset \GL_{n^2}$ from Equation \eqref{gdetneqn}.
The following group plays a central role in the study of regular equivariant
  determinantal representations:
$$
\BG_{\tdet_n,\Lambda_{n-1}}= \{ g\in \BG_{\tdet_n}\mid g\cdot
\Lambda_{n-1}=\Lambda_{n-1}\}.
$$
Let  $\BH\subset\BC^n$ denote the image of $\Lambda_{n-1}$ and  
$\ell_1\in\BC^n$ its kernel. Write $\ell_2$ for $\ell_1$ in the target
$\BC^n$. 
Then $\Mcc_{n}(\BC)=(\ell_1\op \BH)^*\ot (\ell_2\op \BH)$.   
 Let $\transp\in\GL(\Mcc_n(\BC))$ denote
the transpose.\\

\begin{lemma}\label{lem:Gdl} The group $\BG_{\tdet_n,\Lambda_{n-1}}$ is 
$$
\{M\mapsto
\begin{pmatrix} \l_2 & 0\\ v_2& g\end{pmatrix}
M \begin{pmatrix} 1 & \phi_1\\ 0& g\end{pmatrix}^{-1}
\mid g\in\GL(\BH), v_2\in \BH, \phi_1\in \BH^*,\l_2\in \BC^*\}\cdot
\langle\transp\rangle.
$$
In particular, it is isomorphic to  
$$
[\GL(\ell_2)\times \GL(\BH)\ltimes (\BH\op \BH^*\otimes\ell_2)]\rtimes
\BZ_2.
$$
\end{lemma}

\begin{proof}
First note that $\transp(\Lambda_{n-1})=\Lambda_{n-1}$, so it is sufficient to
determine the stabilizer of $\Lambda_{n-1}$ in $\BG_{\det_n}^\circ$. 
Let $A,B\in \GL_n(\CC)$ such that $A\Lambda_{n-1} B^{-1}=\Lambda_{n-1}$.
Since $A$ stabilizes the image of $\Lambda_{n-1}$ and $B$ stabilizes the
Kernel of $\Lambda_{n-1}$ we have
$$
A=\begin{pmatrix} \l_2 & 0\\ v_2& g_2\end{pmatrix}\qquad{\rm and}\qquad
B=\begin{pmatrix} \l_1 & \phi_1\\ 0& g_1\end{pmatrix},
$$
for some $\l_1\in \GL(\ell_1)$,  $\l_2\in \GL(\ell_2)$,
$g_1,\,g_2\in\GL(\BH)$, $v_2\in \BH$ and $\phi_1\in \BH^*$.
The identity $A\Lambda_{n-1} B^{-1}=\Lambda_{n-1}$ is now
equivalent to $g_1=g_2$.
Multiplying $A$ and $B$ by $\lambda_1^{-1}\Id_n$ gives the result.
\end{proof}

\subsection{Facts about complex algebraic groups}\label{alggpfacts}
Let $G$ be an affine complex algebraic group. The group $G$ is 
\begin{itemize}
\item {\it reductive} if
every $G$-module may be decomposed into a direct sum of irreducible $G$-modules.
\item {\it unipotent} if it is isomorphic to a subgroup of the group
  $U_n$ of  upper triangular matrices with 1's on the
  diagonal. 
\end{itemize}
Given a complex algebraic group $G$, there exists a 
maximal normal unipotent subgroup $R^u(G)$, called the {\it unipotent
  radical}. 
The quotient $G/R^u(G)$ is reductive. Moreover there exists subgroups
$L$ in $G$ such that $G=R^u(G)L$. In particular such $L$ are reductive. Such
a subgroup $L$ is not unique, but any two such are conjugate  in $G$ (in fact 
  by an element of $R^u(G)$).
Such a subgroup $L$ is called a {\it Levi factor of $G$}. A good
reference is \cite[Thm. 4. Chap.~6]{MR1064110}.

Malcev's theorem  (see, e.g.,  \cite[Thm.~5. Chap.~6]{MR1064110})
states that fixing a Levi subgroup $L\subset G$ and given any reductive subgroup $H$ of 
$G$, there exists $g\in R^u(G)$ such that $gHg\inv \subseteq L$.

For example, when $G$ is a parabolic subgroup, e.g.
$G=\begin{pmatrix} * & * \\ 0&*\end{pmatrix}$, we have
$L=\begin{pmatrix} * & 0 \\ 0&*\end{pmatrix}$ and
$R^u(G)= \begin{pmatrix} \Id_a & * \\ 0&\Id_b\end{pmatrix}$.

\medskip

A more important example for us is
$R^u(\BG_{\tdet_n,\Lambda_{n-1} })=(\BH\op \BH^*\ot \ell_2)$ 
and a Levi subgroup is $L= (\GL(\ell_2)\times \GL(\BH))\rtimes \BZ_2$. 

\subsection{Outline of the proofs of lower
  bounds}\label{outlinesubsect}   
Let $P\in S^mV^*$ be either a quadric, a permanent or a determinant. 
Say a regular representation $\tilde A$ respects some $G\subseteq \BG_P$.
We may assume that $\tilde A(0)=\Lambda_{n-1}$.

The first step consists in lifting $G$ to  $\BG_A$. 
More precisely, in each case  we construct  a reductive subgroup $\tilde
G$ of $\BG_A$ such that $\bar\rho_A : \tilde G\longto G$ is finite and
surjective. 
In a first reading, it is relatively harmless to  assume that $\tilde G\simeq G$. 
Then, using Malcev's theorem, after possibly  conjugating $\tilde A$, we may
assume that $\tilde G$ is contained in $(\GL(\ell_2)\times
\GL(\BH))\rtimes \BZ_2$.
Up to considering an index two subgroup of $\tilde G$ if necessary, we
assume that $\tilde G$ is contained in $\GL(\ell_2)\times
\GL(\BH)$ (with the notation of Lemma~\ref{lem:Gdl}).

Now, both $\Mcc_n(\BC)= (\ell_1\op \BH)^*\ot (\ell_2\op \BH)$ and 
$V$ (via $\bar\rho_A$) are  $\tilde G$-modules. Moreover, $A$ is an
equivariant embedding of $V$ in  $\Mcc_n(\BC)$. This turns out to
be a very restrictive condition. 

Write 
$$
\cM_n(\BC)= \begin{pmatrix} \ell_1^*\ot \ell_2 & \BH^*\ot \ell_2\\ 
\ell_1^*\ot \BH & \BH^*\ot \BH\end{pmatrix}, \ \ \
\Lambda_{n-1}= \begin{pmatrix} 0 & 0\\ 0&\Id_\BH\end{pmatrix}.
$$
If $m\geq 2$ the  $\ell_1^*\ot \ell_2$ coefficient  of
$\tilde A$ has to be zero. 
Then, since $P\neq 0$, the projection of $A(V)$ on  $\ell_1^*\ot
\BH\simeq \BH$ has to be non-zero. We thus have
a $G$-submodule $\BH_1\subset \BH$ isomorphic to an irreducible submodule of
$V$. 
A similar argument  shows that there
must be another irreducible $G$-submodule $\BH_2\subset \BH$ such that
an irreducible submodule of $V$ appears in $\BH_1^*\ot \BH_2$.

In each case, we can construct a sequence of irreducible sub-$\tilde
G$-modules $\BH_k$ of $\BH$ satisfying very restrictive
conditions. This allows us to get our lower bounds.

\bigskip

To prove the representations $\tilde A$ actually
compute the polynomials we want, in the case $G=\BG_P$, we first check
that $\BG_P$ is contained in the image of $\bar\rho_A$. Since $P$ is
characterized by its symmetries, we  deduce that $\tdet_n\circ\tilde
A$ is a scalar multiple of $P$. We then specialize to
evaluating on the diagonal matrices  in $\cM_m(\BC)$  to determine this constant, proving in particular that it
is non-zero.



\section{Symmetric determinantal representations of  quadrics}
\label{sec:quad}

 We continue the notation of \S\ref{quadricsres}, in particular $Q\in S^2\BC^{2s}$ is
 a nondegnerate quadric.

Let $\tau\,:\BC^*\times O(2s)\longto GL(2s)$  be given by $ (\lambda,M)\mapsto
\lambda M$. The image of $\tau$ is the group $\BG_Q$and 
$\chi_Q\circ\tau(\lambda,M)=\lambda^{-2}$.

In the expression \eqref{spone} we have  $\BG_{A}= (\GL(\ell_2)\times
\GL(\BH))\rtimes \ZZ/2\ZZ$.
Since  $\rho_A(\BG_A^\circ)$ is a proper subgroup of $\BG_Q$, this determinantal representation is not equivariant.

More generally, let $\tilde B$ be any determinantal representation  of $Q$
of size $s+1$. By Lemma~\ref{lem:vzg} $\trank(\Lambda)=s$. Then the dimension of $\BG_B$ cannot exceed that of
$\BG_{\det_{s+1},\Lambda_s}$, which  is $(s+1)^2$ by Lemma~\ref{lem:Gdl}. 
But the dimension of $\BG_Q$ is $2s^2-s+1$. Hence the representation
cannot respect  the symmetries. (This argument has to be refined
when  $s=3$, observing that the unipotent radical of
$\BG_{\det_{s+1},\Lambda_s}$ is contained in the kernel of $\rho_A$.)

 Nonetheless, in the case of quadrics, the smallest presentation $A$  respects
about \lq\lq half\rq\rq\ the symmetry, as was the case in Example~\ref{optimalwaring}. We will see this again with
$\tperm_3$ but so far have no explanation.

\begin{proof}[Proof of Proposition~\ref{quadprop}]
Let $\tilde A=\Lambda + A$ be a equivariant determinantal representation
of $Q$. 
By Lemma~\ref{lem:vzg},   one may assume that
$\tilde A(0)=\Lambda_{n-1}$.  

We   now construct  an analog $L$  of the group $\tilde G$ mentioned in
\S\ref{outlinesubsect}.
Start with $H=\bar\rho_A\inv(\BG_Q)$. 
Consider a Levi decomposition $H=R^u(H)L$. 
Then $\bar\rho_A(R^u(H))$ is a normal unipotent subgroup of $\BG_Q$. 
Since $\BG_Q$ is reductive this implies that $R^u(H)$ is contained in
the kernel of $\bar\rho_A$. In particular, $\bar\rho_A(L)=\BG_Q$.
Since $\BG_Q$ is connected, $\bar\rho_A(L^\circ)=\BG_Q$.

By construction  $L$ is a reductive subgroup of $\BG_{\tdet_n,\Lambda_{n-1}}$.
 By Malcev's theorem, possibly after conjugating $\tilde A$, we may and will
 assume that $L$ is contained in  $ (  \GL(\ell_2)\times
 \GL(\BH) )\rtimes \BZ_2$. In particular $L^\circ$ is contained in 
$ \GL(\ell_2)\times \GL(\BH) $
and  $\ol{\rho}_A(L^\circ)=\BG_Q$.

Observe that $A(V)$ is  an irreducible   $L^\circ$-submodule of $\Mcc_n(\BC)$ isomorphic to
$V$. Moreover, the action of $L^\circ$  respects the decomposition
$$
\Mcc_n(\BC)=\ell_1^*\ot\ell_2\oplus \ell_1^*\ot \BH\oplus
\BH^*\ot\ell_2\oplus \BH^*\ot \BH.
$$
The projection of $A(V)$ on $\ell_1^*\ot\ell_2$ has to be zero, since
it is $L^\circ$-equivariant. Hence
$$
A(V)\subset \ell_1^*\ot \BH\oplus
\BH^*\ot\ell_2\oplus \BH^*\ot \BH.
$$
In matrices:
$$
A(V)\subset \begin{pmatrix} 0 & \BH^*\ot \ell_2\\ \ell_1^*\ot \BH & \BH^*\ot \BH\end{pmatrix}.
$$
  Thus for the determinant
to be non-zero, we need the projection to $\ell_1^*\ot \BH$ to be non-zero. Thus
it must contain at least one copy of $V$. In particular
$
\dim(\BH)\geq\dim(V)
$;   the desired inequality.
\end{proof}


\section{Proofs of the determinantal representations of the determinant} 
\label{sec:det}
Recall our notations that $E,F$ are complex vector  spaces of dimension $m$
and we have an identification $E\otimes F^*\simeq \Hom(F,E)$.
A linear map $u: F\ra E$ induces linear maps
\begin{align}
\label{uwk} u^{\ww k} : \La k F&\ra \La k E\\
\nonumber v_1\ww\cdots \ww v_k&\mapsto u(v_1)\ww\cdots \ww u(v_k).
\end{align}
In   the case $k=m$, $u^{\ww m}$ is called the determinant of $u$ and we denote it
$\tDet(u)\in \Lambda^mF^*\ot \Lambda^mE$. 
 The map 
\begin{align*}
E\otimes F^*=\Hom(F,E)&\longto \Lambda^mF^*\ot \Lambda^mE\\
  u&\longmapsto \tDet(u) 
\end{align*}
is polynomial, homogeneous of degree $m$,  and equivariant for the
natural action  of $\GL(E)\times \GL(F)$.                                                                                                                                                                                                                                            

 The transpose of $u$ is 
 \begin{align*} u^T : E^*&\longto
 F^*,\\
 \varphi&\longmapsto\varphi\circ u.
 \end{align*}
 Hence $u^T\in F^*\otimes E$ is
 obtained from $u$ by switching $E$ and $F^*$,
 and 
$\tDet(u^T)\in \Lambda^mE\ot \Lambda^mF^*$.  Moreover,
$\tDet(u^T)=\tDet(u)^T$.

 \begin{proof}[Proof of Proposition~\ref{dethalfrep}]
Set $P=\tdet_n\circ\tilde A$.
To analyze the action of $\GL(E)$ on $\tilde A$,   reinterpret
$\BC^{n*}\ot \BC^n$ without the identification $\Lambda^0E\simeq\Lambda^mE$
as $(\oplus_{j=0}^{m-1}\Lambda^jE)^*\ot ( \oplus_{i=1}^{m}\Lambda^iE)$.

For each $u\in E\ot F^*$,  associate to $\tilde A(u)$  a linear map $\tilde a(u) : \oplus_{j=0}^{m-1}\Lambda^jE\ra \oplus_{i=1}^{m}\Lambda^iE$.
Then $\tDet(\tilde a(u))\in \Lambda^{n}(\oplus_{j=0}^{m-1}\Lambda^jE^*)\ot
\Lambda^{n}(\oplus_{i=1}^{m}\Lambda^iE)$. This space may be canonically identified  as a $GL(E)$-module  
with $\Lambda^0E^*\ot \Lambda^mE\simeq \Lambda^mE$. 
(The identification $\Lambda^0E\simeq\Lambda^mE$ allows one  to identify
this space with $\CC$.) Using the maps  \eqref{uwk}, 
we get $\GL(E)$-equivariant maps

\begin{center}
\begin{tikzpicture}
  \matrix (m) [matrix of math nodes,row sep=3em,column sep=4em,minimum width=2em]
  {E\otimes F^*&\Hom(\oplus_{j=0}^{m-1}\Lambda^jE,\oplus_{i=1}^{m}\Lambda^iE)&\Lambda^mE.\\};
  \path[-stealth]  
    (m-1-1) edge node [above] {$\tilde a$} (m-1-2)
    (m-1-2) edge node [above] {$\tDet$} (m-1-3);       
\end{tikzpicture} 
\end{center}
Hence for all $u\in E\otimes F^*$ and all $g\in\GL(E)$,  
\begin{eqnarray}
  \label{eq:3}
  \begin{array}{l@{\,}l}
\  \tDet(\tilde a(g^{-1}u))&=(g\cdot\tDet)(\tilde a(u))\\
&=\det(g)\inv\tDet(\tilde a(u)).
  \end{array}
\end{eqnarray}

Equation~\eqref{eq:3} shows that $\GL(E)$
is contained  in the image of $\bar \rho_A$.\\

 Equation~\eqref{eq:3} also proves that $P$ 
is a scalar (possibly zero) multiple of the determinant.
Consider $P(\Id_m)=\det_n(\tilde A(\Id_m))$. Perform a Laplace expansion
of this large determinant:  there is only one non-zero expansion term, so $P$ is
the determinant up to a sign.

To see the sign is as asserted in Proposition~\ref{dethalfrep}, specialize to the diagonal matrices, then there is just one term. Note that
$y^i_i$ appears in the large matrix with the sign $(-1)^{i+1}$. Thus the total contribution
of these signs is $(-1)$ if $m\equiv 2,3\tmod 4$ and $(+1)$ if $m\equiv 0,1\tmod 4$.

The $k$-th  block of $\Lambda$ contributes a sign of  $(-1)^{\binom mk-1}$ if we perform a left to
right Laplace expansion, except for the last which contributes $(1)^{m-1}$.
(The terms are   negative because the $Id_{\binom mk -1}$ will begin
in the left-most column each time, but, except for the last block, 
it begins in the second row.)
Thus the total contribution of $\Lambda$ to the sign is $(-1)^{\sum_{k=1}^{m-2}[\binom mk -1]}=1$.

 The slot of  each $\pm y^i_i$ except $y^m_m$ (whose slot   always contributes positively in the Laplace expansion)  
contributes a $(-1)^{\binom m{i-1} }$ (it is always left-most, but appears $\binom m{i-1}+1$  slots below
the top in the remaining matrix). Thus the total contribution from these slots is $(-1)^{\sum_{i=0}^{m-2} \binom mi}=
(-1)^{2^m-m-1}=(-1)^{m+1}$. 

Thus the total sign is $-1$ if $m\equiv 0, 3\tmod 4$ and $+1$ if $m\equiv 1,2\tmod 4$. 
\end{proof}

\begin{proof}[Proof of Theorem~\ref{th:srdcdet}] 
Let $\tilde A=\Lambda + A: \Mcc_m(\BC)\ra \Mcc_n(\BC)$  be a regular determinantal representation of $\tdet_m$
that respects $\GL(E)$.  It remains
to prove that $n\geq 2^m-1$.

We may assume $\Lambda =\Lambda_{n-1}$.
As in the proof of Proposition~\ref{quadprop}, after possibly  conjugating
$\tilde A$, we construct a connected reductive subgroup $L$ of
$\GL(\ell_2)\times\GL(\BH)$ mapping onto $\GL(E)$ by $\bar\rho_A$. 


We have an action of $L$ on $\cM_n(\BC)$, but  we would like to work
with $\GL(E)$. Towards this end, there exists a finite cover $\tau :
\tilde L\longto L$ 
that is isomorphic to the product of a torus and a product of simple
simply connected groups.  
In particular there exists a subgroup of $\tilde L$ isomorphic to
$\CC^*\times \SL(E)$
such that $\bar\rho_A\circ\tau(\CC^*\times \SL(E))=\GL(E)$.
The group $\CC^*\times \SL(E)$ acts trivially on
$\ell_1$,  on
 $\ell_2$ (by some character) and  on $\BH$. It acts on $\Mcc_n(\BC)=(\ell_1^*\oplus
\BH)\ot(\ell_2\oplus \BH) $ accordingly.

\medskip

The $\CC^*\times
\SL(E)$-module
  $A(V)$ is isomorphic to the sum of $m$ copies of $E$, and $E$ is an irreducible $\CC^*\times \SL(E)$-module.
In particular its equivariant projection on   $\ell_1^*\ot\ell_2$ is zero, 
which implies that the $(1,1)$ entry of the matrix of  $\tilde A$ (in adapted bases) is zero.  
Consider the  equivariant projection of $A(V)$  on    $\ell_1^*\ot \BH$.
This projection in bases goes to the remainder of the first column. It must be non-zero or
$\tdet_n\circ\tilde A$ will be identically zero.
Since it  is equivariant, $\ell_1^*\ot \BH\simeq \BH$ must contain $E$
as a  $\CC^*\times \SL(E)$-module.
Similarly, examining the first row,    $\BH^*\ot \ell_2$ has to contain $E$
as a  $\CC^*\times \SL(E)$-module.

If $m=2$, it is possible that $\BH\simeq E$ and $\BH\otimes \ell_2\simeq
E^*$.
In this case, $\tdet_2$ is a quadratic form, and we recover its
determinantal representation of size 3. 

Assume now that $m\geq 3$, in particular that $E$ and $E^*$ are not
isomorphic as   $\SL(E)$-modules. 
We just proved that $\BH$ must contain a subspace isomorphic to $E$,  say
$\BH_1$. Since $\BH_1^*\ot \ell_2$ is   an irreducible $\SL(E)$-module and not
isomorphic to $E$, the projection of $A(V)$ on this
factor is zero. 

Choose a $\BC^*\times\SL(E)$-stable complement $S_1$ to $\BH_1$ in $\BH$.
If the projection of $A(V)$ to  the block $\BH_1^*\ot S_1$ is zero, 
  by expanding the columns corresponding to $\BH_1^*$, one sees that
$\det_m$ is equal to the determinant in $(\ell_1\oplus
S_1)^*\ot(\ell_2\ot S_1)$, and we can restart the proof with $S_1$ in place of $\BH$. 

So assume  that  the projection of $A(V)$ onto  the block $\BH_1^*\ot S_1$ is non-zero.
 Then there  must be some irreducible $(\CC^*\times \SL(E))$-submodule $\BH_2$ such that
 $\BH_1^*\ot \BH_2$ contains $E$ as a submodule.
Continuing, we get a sequence of simple $(\CC^*\times
\SL(E))$-submodules  $\BH_1,\dots,\BH_k$ of $\BH$ such that $E$ is a submodule
of $\BH_i^*\ot\BH_{i+1}$ and of $\ell_2\ot\BH_k^*$.



 The situation is easy to visualize with Young diagrams.
 The  irreducible polynomial 
 representations of $\GL(E)$ are indexed by partitions, and irreducible representations
 of $\SL(E)$ by equivalence classes of partitions where $\pi\pm  (c^m)\sim \pi$, and to a partition $\pi=(p_1\hd p_{m-1})$
 we associate a Young diagram, a collection  of left-aligned boxes with
 $p_j$ boxes in the $j$-th row.
 For example, the Young diagram for $\pi=(2,1,1)=:(2,1^2)$ is
 $$
 \yng(2,1,1). 
 $$
 For example, as an $SL(E)$-module,  $E$ (resp. $E^*$) corresponds the class of to a single box (resp. a column of
 $m-1$ boxes). (As a $GL(E)$-module, $E^*$ corresponds to a diagram with $-1$ boxes.)
  The Pieri formula  implies that $E\subset S_{\pi}E^*\ot S_{\mu}E$ if and only if the diagram
of $\mu$ is obtained from the diagram of $\pi$ by adding a box. 
Then the sequence of Young diagrams associated to  the irreducible $\SL(E)$  modules $\BH_i$ start
with one box,   and increases by one
box at each step. Thus we must have $\BH_k$ associated to $\pi=(c^{m-1},c-1)$ for some $c$.
 To have the proper $\BC^*$-action, we choose the action on $\ell_2$ to cancel
 the $(c-1)\times m$ box. 
 We deduce that
$(\BH_1,\cdots,\BH_k)=(\Lambda^1E,\Lambda^2E,\dots,\Lambda^{m-1}E\simeq
E^*)$ is the unique minimal sequence of   modules. In particular the dimension of $\BH$ is
at least  $\sum_{k=1}^{m-1}\binom m k =2^m-2$.
\end{proof}

\begin{proof}[Proof of Proposition~\ref{detwithrespect}]
Write $\tilde A$ as
$$
\begin{array}{cccl}
  \tilde a\,:&E\ot F^*&\longto&\bigg(\oplus_{i=1}^m\Lambda^iE\ot\Lambda^iF^*\bigg)\ot \bigg(\oplus_{j=0}^{m-1}\Lambda^jE\ot\Lambda^jF^*\bigg)^*.
\end{array}
$$
For $u\in E\ot F^*$, $\tDet(\tilde a(u))$ belongs to 
$$
\bigg(\Lambda^n\oplus_{j=0}^{m-1}\Lambda^jE\ot\Lambda^jF^*\bigg)^*\ot \bigg(\Lambda^n\oplus_{i=1}^{m}\Lambda^iE\ot\Lambda^iF^*\bigg),
$$
which may be canonically identified  with
$$
(\Lambda^0E\ot\Lambda^0F^*)^*\ot\Lambda^mE\ot\Lambda^mF^*\simeq \Lambda^mE\ot\Lambda^mF^*
$$
These identifications determine  $\GL(E)\times\GL(F)$-equivariant polynomial maps
\begin{center}
\begin{tikzpicture}
  \matrix (m) [matrix of math nodes,row sep=3em,column sep=3em,minimum width=2em]
  {E\ot F^*&\bigg(\oplus_{i=1}^m\Lambda^iE\ot\Lambda^iF^*\bigg)\ot \bigg(\oplus_{j=0}^{m-1}\Lambda^jE\ot\Lambda^jF^*\bigg)^*
&
\Lambda^mE\ot\Lambda^mF^*.\\};
  \path[-stealth]  
    (m-1-1) edge node [above] {$\tilde a$} (m-1-2)
    (m-1-2) edge node [above] {$\tDet$} (m-1-3);       
\end{tikzpicture} 
\end{center}

Choose  bases $\Bc_E$ and $\Bc_F$, respectively of $E$ and $F$ to identify $E\ot F^*$  with
$\Mcc_m(\CC)$.  
Choose a total order on the subsets of $\Bc_E\times\Bc_F$, to get
bases of each $\Lambda^jE\ot\Lambda^jF^*$ and maps 

\begin{center}
\begin{tikzpicture}
  \matrix (m) [matrix of math nodes,row sep=3em,column sep=4em,minimum width=2em]
  {\Mcc_m(\CC)&\Mcc_n(\CC)&\CC.\\};
  \path[-stealth]  
    (m-1-1) edge node [above] {$\tilde A$} (m-1-2)
    (m-1-2) edge node [above] {$\tdet_n$} (m-1-3);       
\end{tikzpicture} 
\end{center}

As in the proof of Proposition~\ref{dethalfrep}, this implies that
$\GL(E)\times\GL(F)$ belongs to the image of $\bar\rho_A$, so
  $P=\tdet_n\circ\tilde A$ is a scalar multiple of $\tdet_m$.\\

To see it is the correct multiple, specialize to the diagonal matrices. Re-order the
rows and columns so that all non-zero entries of $  A(\cM_m(\BC_m))$ appear in the upper-left
corner. Note that since we made the same permutation to rows and columns this does not change
the sign. Also note that since we have diagonal matrices, there are only plus signs for the entries
of $  A(\cM_m(\BC_m))$. In fact this upper-left corner is exactly Grenet's representation for
$$\tperm_m
\begin{pmatrix} y^1_1 & \cdots & y^m_m\\
& \vdots &\\ 
 y^1_1 & \cdots & y^m_m
\end{pmatrix}
$$
which is $m!(y^1_1\cdots y^m_m)$. Finally note that each term in an expansion contains $n-m$ elements
of $\Lambda_0$ to conclude.\\

It remains to prove that $\transp$ belongs to the image of
$\bar\rho_A$. The following diagram is commutative:
\begin{center}
\begin{tikzpicture}
  \matrix (m) [matrix of math nodes,row sep=3em,column sep=3em,minimum width=2em]
  {E\ot F^*&
\bigg(\oplus_{i=1}^m\Lambda^iE\ot\Lambda^iF^*\bigg)\ot \bigg(\oplus_{j=0}^{m-1}\Lambda^jE\ot\Lambda^jF^*\bigg)^*
&
\Lambda^mE\ot\Lambda^mF^*.\\
F^*\ot E&
\bigg(\oplus_{i=1}^m\Lambda^iF^*\ot\Lambda^iE\bigg)\ot
\bigg(\oplus_{j=0}^{m-1}\Lambda^jF^*\ot\Lambda^jE\bigg)^*
&
\Lambda^mF^*\ot\Lambda^mE,\\
};
  \path[-stealth]  
    (m-1-1) edge node [above] {$\tilde a$} (m-1-2)
    (m-1-2) edge node [above] {$\tDet$} (m-1-3)
    (m-2-1) edge  (m-2-2)
    (m-2-2) edge node [above] {$\tDet$} (m-2-3)
    (m-1-1) edge node [right] {$\transp$} (m-2-1)
    (m-1-2) edge node [right] {$\Delta\transp$} (m-2-2)
    (m-1-3) edge node [right] {$\transp$} (m-2-3);   
\end{tikzpicture} 
\end{center}
where $\Delta\transp$ is the transposition on each summand. 

Using $\Bc_E$ and $\Bc_F$, we identify the 6 spaces with  matrix
spaces. The first vertical map becomes the transposition from
$\Mcc_m(\CC)$ to itself. 
The last vertical map becomes the identity on $\CC$.
The middle vertical map  is the endomorphism of $\Mcc_n(\CC)$
corresponding to bijections between bases of spaces
$\Lambda^jE\ot\Lambda^jF^*$ and  $\Lambda^jF^*\ot\Lambda^jE$. 
It follows that there exist two permutation matrices
$B_1,B_2\in\GL_n(\CC)$ such that
$$
\forall M\in\Mcc_m(\CC)\qquad 
\tilde A(M^T)=B_1\tilde A(M)B_2\inv, 
$$
proving that $\transp$ belongs to the image of $\bar\rho_A$.
\end{proof}

\section{Proofs of results on determinantal representations of $\tperm_m$} 
\label{sec:perm}

\begin{proof}[Proofs of Propositions~\ref{grenetrep} and~\ref{permwithrespect}]
The maps $s_k(v): (S^kE)_{reg} \ra (S^{k+1}E)_{reg}$
are related to the maps 
  $ex_k(v): \La k E\ra \La{k+1}E$ 
 as follows. The sources of both maps have bases indexed by multi-indices $I=(i_1\hd i_k)$ with
 $1\leq i_1<\cdots < i_k\leq m$, and similarly for the targets.
The maps are the same on these basis vectors except for with $s_k(v)$ all the coefficients are
positive whereas with $ex_k(v)$ there are signs. Thus the polynomial computed by
\eqref{Grenetexpr} is the same as the polynomial computed by \eqref{regdet} except
all the $y^i_j$ appear positively. Reviewing the sign calculation, we get the result.

The maps $S_k$ and $EX_k$ are similarly related and we conclude this case similarly.
\end{proof}  
  
\begin{remark} The above proof can be viewed more invariantly in terms of
the Young-Howe duality functor described in \cite{MR2308168}.
\end{remark}

\begin{proof}[Proof of Theorem~\ref {halfsdcperm}]
Write  $E,F=\CC^m$.
Let $\tilde A$ be a determinantal representation of $\tperm_m$ such
that $\tilde A(0)=\Lambda_{n-1}$.
Embed $N(T^{\GL(E)})$ in $\GL(\Hom(F,E))$ by $g\longmapsto
\{M\mapsto gM\}$.
We assume  the image of $\bar\rho_A$ contains $N(T^{\GL(E)})$. 
Set $N=N(T^{\GL(E)})$ and $T=T^{\GL(E)}$.

As in the proof of Theorem~\ref{halfsdcdet}, we get a reductive
subgroup $L$ of $(\GL(\ell_2)\times \GL(\BH))\rtimes\ZZ_2$ mapping
onto $N$ by $\bar\rho_A$.
In the determinant case, at this point we dealt with
the universal cover of the connected reductive group $\GL(E)$.
Here the situation is more complicated for two reasons.
First,  there is no \lq\lq finite universal cover\rq\rq\  of    $\Symgr_m$ (see e.g. \cite{MR991411,MR1007885}).
Second, since our group is not connected, we will have to deal with  the factor $\BZ_2$ coming from
transposition, which will force us to work with a subgroup of
$N$. Fortunately this will be enough for our purposes.


We first deal with the $\BZ_2$:
Since $L/(L\cap \BG_{\tdet_n,\Lambda_{n-1}}^\circ)$ embeds in
$\BG_{\tdet_n,\Lambda_{n-1}}/\BG_{\tdet_n,\Lambda_{n-1}}^\circ\simeq\BZ_2$, the subgroup 
$L\cap \BG_{\tdet_n,\Lambda_{n-1}}^\circ$ has index 1 or 2 in $L$.  
Since the alternating group $\Altgr_m$ is the only index 2 subgroup of $\Symgr_m$, 
$\bar\rho_A(L\cap \BG_{\tdet_n,\Lambda_{n-1}}^\circ)$ contains $T\rtimes\Altgr_m\subset N $.
In any case, there exists a reductive subgroup $L'$ of $L$   such that
$\bar\rho_A(L')=T\rtimes\Altgr_m\subset N$.\\

Next we deal with the lack of a lift. We will get around this by showing we may 
  label irreducible $L'$ modules only using
labels from 
$\bar\rho_A(L')=T^{\GL(E)}\rtimes\Altgr_m$. 

The connected reductive group $L'^\circ$ maps onto $T$.
Let  $Z$ denote the  center of $L'^\circ$; then $\bar\rho_A(Z^\circ)=T$. 
In particular the character group $X(T)$ may be identified  with a subgroup of the character group  $X(Z^\circ)$.
The action of $L'$ by conjugation   on itself induces an action of the finite group $L'/L'^\circ$ on $Z^\circ$. 
Moreover, the morphism $\bar\rho_A$ induces a surjective map  $\pi_A\,:\,L'/L'^\circ\longto \Altgr_m$. 
These actions are compatible  in the sense that
for  $g\in L'/L'^\circ$, $t\in T$, $z\in Z^\circ$ and $\sigma\in
\Altgr_m$ satisfying  $\pi_A(g)=\sigma$ and $\bar\rho_A(z)=t$, 
$$
\sigma\cdot t=\bar\rho_A(g\cdot z).
$$
In particular, both  the kernel of ${\bar\rho_A}$ restricted to ${Z^\circ}$ and $X(T)$ are stable under  the $(L'/L'^\circ)$-action.

Let $\Gamma_\QQ$ be a complement of the subspace $X(T)\ot \QQ$ in 
the vector space $X(Z^\circ)\ot \QQ$ stable under  the action of
$(L'/L'^\circ)$. 
Set $\Gamma=\Gamma_\QQ\cap X(Z^\circ)$ and 
$\tilde T=\{t\in Z^\circ\,:\,\forall \chi\in
\Gamma\qquad\chi(t)=1\}$. 
Then $\tilde T$ is a subtorus of $Z^\circ$ and the restriction of $\bar\rho_A$ to $\tilde T$ is a finite morphism onto $T$.

The character group $X(\tilde T)$ may be identified  with $X(Z^\circ)/\Gamma$ by restriction.
 Then $X(T)$ may be identified with a subgroup of $X(\tilde T)$ of finite index.
 Hence there exists a natural number $k_0$ such that 
$k_0X(\tilde T)\subset X(T)\subset X(\tilde T)$. 

Let   $W$ be an irreducible representation of $L'$. It decomposes under the action of $\tilde T$ as  
$$
W=\oplus_{\chi\in X(\tilde T)}W^\chi.
$$
Set $\Wt(\tilde T,W)=\{\chi\in X(\tilde T)\,:\,W^\chi\neq\{0\}\}$. 
The group $L'$ acts by conjugation on $\tilde T$ and so on $X(\tilde
T)$. 
By the rigidity of tori (see e.g.,  \cite[\S16.3]{MR0396773}), $L'^\circ$
acts trivially. 
Hence, the finite group
 $L'/L'^\circ$ acts on $\tilde T$ and $X(\tilde T)$.
For any $t\in \tilde T$, $h\in L'$ and $v\in W$,
$thv=h(h^{-1}th)v$. 
Hence  $\Wt(\tilde T,W)$ is stable under  the action of $L'/L'^\circ$.  
For $\chi\in\Wt(\tilde T,W)$, the set  $\oplus_{\sigma\in L'/L'^\circ}W^{\sigma\cdot\chi}$ is stable under  the action $L'$. 
By irreducibility of $W$, one deduces that $\Wt(\tilde T,W)$ is a single  $(L'/L'^\circ)$-orbit.
Then $k_0\Wt(\tilde T,W)\subset X(T)$ is an  $\Altgr_m$-orbit. 


We are now in a position to argue as in \S\ref{outlinesubsect}.

\medskip

Let $\varepsilon_i$ denote the character of $T$ that maps an element
of $T$ on its i$^{th}$ diagonal entry.
The set $\{a_1\varepsilon_1+\dots+a_m\varepsilon_m\,:\,
a_1\geq\cdots\geq a_{m-1}\quad{\rm and}\quad a_{m-2}\geq a_m\}$ is a
fundamental domain of the action of $\Altgr_m$ on $X(T^{\GL(E)})$.
Such a weight is said to be $\Altgr_m$-dominant. 
Hence, there exists a unique  $\Altgr_m$-dominant weight $\chi_W$ such
that $k_0\Wt(\tilde T,W)=\Altgr_m.\chi_W$. 

\medskip

Summary of the properties of $L'$, $\tilde T$ and $k_0$:
\begin{enumerate}
\item $L'$ is a reductive subgroup of $\GL(\ell_2)\times\GL(\BH)$;
\item $\bar\rho_A(L')=\Altgr_m\ltimes T$ ;
\item $\tilde T$ is a central subtorus of $(L')^\circ$; 
\item $\bar\rho_A\,:\,\tilde T\longto T$ is finite and surjective, inducing embeddings $k_0X(\tilde
  T)\subset X(T)\subset X(\tilde T)$;
\item For any irreducible representation $W$ of $L'$ there exists a
  unique $\Altgr_m$-dominant weight $\chi_W$ such
that $k_0\Wt(\tilde T,W)=\Altgr_m.\chi_W$.
\item For the standard representation $E$ of $L'$ through  $\bar\rho_A$, $\chi_E=k_0\varepsilon_1$.
\end{enumerate}

\bigskip
The action of $L'$ on $\Mcc_n(\BC)=(\ell_1\op \BH)^*\ot (\ell_2\op
\BH)$ 
respects the decomposition
$$
\Mcc_n(\BC)=\ell_1^*\ot\ell_2\oplus \ell_1^*\ot \BH\oplus
\BH^*\ot\ell_2\oplus \BH^*\ot \BH.
$$

The image of $A$ is an $L'$-module isomorphic to the sum of $m$ copies
of $E$.
In particular, its projection on $\ell_1^*\ot\ell_2$ has to be
zero. Hence
$$
A(V)\subset \ell_1^*\ot \BH\oplus
\BH^*\ot\ell_2\oplus \BH^*\ot \BH.
$$
As was the case before,  for the determinant
to be non-zero, we need the projection to $\ell_1^*\ot \BH$ to be non-zero, so
it must contain at least one copy $\BH_1$ of $E$.

Assume first that $\BH_1^*\ot\ell_2\simeq E$. This happens only if
$m=2$, where $\perm_m$ is a quadric.

Assume now that  $\BH_1^*\ot\ell_2\not\simeq E$.
Choose an $L'$-stable complement $\BS_1$ of $\BH_1$ in $\BH$.
If the projection of $A(V)$ on $\BH_1^*\ot\BS_1$ is zero, one can
discard $\BH_1$ and start over as in the proof of the determinant case,
so we assume it contributes non-trivially. 
Continuing so on, one gets a sequence $\BH_1,\dots,\BH_k$ 
of irreducible $L'$-submodules
of $\BH$ in direct sum such that
\begin{enumerate}
\item $k\geq 2$;
\item $E\subset \BH_i^*\ot\BH_{i+1}$ for any $i=1,\dots,k-1$;
\item $E\simeq \BH_k^*\ot\ell_2$.
\end{enumerate}

Let $\gamma\,:\,\CC^*\longto \tilde T$ be a group homomorphism such that $\bar\rho_A\circ\gamma(t)=t^{k_0}\Id_E$.
Then $\g(\CC^*)$ acts trivially   on $\BH_1^*\ot \BH_1$ and
with weight $k_0$ on $E$. Hence the projection of 
$A(V)$ on $\BH_1^*\ot \BH_1$ is zero.
(Recall that
  via $\Lambda$, $\Id_{\BH_1}\in \BH_1^*\ot \BH_1$.) 
More generally the action of $\gamma$ shows  that the non-zero
blocks of $A(V)$ are   $\ell_1^*\ot \BH_1$,
$\BH_i^*\ot\BH_{i+1}$, and $\BH_k^*\ot\ell_2$.
Consider the following
picture:
$$
\begin{pmatrix}
0&\BH_1^*\ot \ell_2 & ...\\
E& \Id_{\BH_1}&...\\
\vdots & \BH_1^*\otimes S_1 & ...\end{pmatrix}.
$$

Write
$$
k_0\Wt(\tilde T,\BH_i)=\Altgr_m\cdot  \chi_{\BH_i}.
$$
Then
$$
k_0\Wt(\tilde T,\BH_i^*\ot
\BH_{i+1})=\{-\sigma_1\chi_{\BH_i}+\sigma_2\chi_{\BH_{i+1}}\,:\,\sigma_1,\sigma_2\in\Altgr_m\}.
$$
This set has to contain $k_0\Wt(\tilde
T,E)=\{k_0\varepsilon_i\,\mid \,i\in [m]\}$.
We deduce that $\chi_{\BH_{i+1}}=\sigma\chi_{\BH_i} +k_0 \varepsilon_u$ for some
$\sigma\in\Altgr_m$ and $u\in[m]$.

We define the length $\ell(\chi)$ of $\chi\in X(T)$ as its number of
non-zero coordinates in the basis $(\varepsilon_1,\dots,\varepsilon_m)$.
Then 
\begin{eqnarray}
  \label{ineq:lengthchi}
  \ell(\chi_{\BH_{i+1}})\leq \ell(\chi_{\BH_{i}})+1.
\end{eqnarray}




Observe that $\chi_{\ell_2}$ is invariant under $\Altgr_m$. 
Hence $\chi_{\ell_2}=\a(\varepsilon_1+\cdots +\varepsilon_m)$
for some $\a\in \BZ$. 
The action of $\gamma$ shows that $\a=k-1$.
We deduce that $\ell(\chi_{\BH_k})\geq m-1$.
Then, by inequality~\eqref{ineq:lengthchi}, there exists a subset $\BH_{i_1},\dots, \BH_{i_{m-1}}$ of the
$\BH_j$'s with $\ell(\chi_{\BH_{i_s}})=s$.

We claim that $\dim(\BH_{i_s})\geq \binom m s$.
First, $\dim(\BH_{i_s})$ is greater or equal to the cardinality of
$\Altgr_m \chi_{\BH_{i_s}}$. Since $m\geq 3$, $\Altgr_m$ acts
transitively on the subsets of $[ m ]$ with $s$ elements.
The claim follows.

Summing these inequalities on the dimension of the $\BH_{i_s}$, we get
$$
\dim \BH\geq\sum_{j=0}^{m-1}\dim \BH_{i_j}\geq\sum_{j=1}^{m-1}\binom
mj=2^m-2.
$$
\end{proof}

\begin{proof}[Proof of Theorem~\ref{mainthm}]
 This proof is omitted since it is very similar to the proof  of Theorem~\ref{th:srdcdet}.
\end{proof}

\bibliographystyle{amsalpha}

\bibliography{LRpermdetNR}

\end{document}